\documentclass[11pt,a4paper]{article}
\usepackage[utf8]{inputenc}
\usepackage[T1]{fontenc}

\newcommand{\Stab}{\operatorname{Stab}} %the stabilizer (isotropy) subgroup
\newcommand{\diag}{\operatorname{diag}} %the diagonal embedding
\newcommand{\und}[1]{\underline{#1}}
\def\zone{{\{0, 1\}^n}}
\usepackage{amsmath}
\usepackage{amsthm}
\usepackage{amssymb}
\usepackage{amscd}
\usepackage{graphicx}
\usepackage{epsfig}
\usepackage[matrix,arrow,curve]{xy}
\usepackage{color}
\usepackage{mathrsfs}

\usepackage{bbm}
\usepackage{stmaryrd}
\usepackage{hyperref}

\usepackage{latexsym}
\usepackage{amsfonts}
\input xy
\xyoption{all} \tolerance=500
\newtheorem{theorem}{Theorem}[section]
\newtheorem{proposition}[theorem]{Proposition}
\newtheorem{lemma}[theorem]{Lemma}

\theoremstyle{definition}
\newtheorem{example}[theorem]{Example}
\newtheorem{remark}[theorem]{Remark}
\newtheorem{definition}[theorem]{Definition}

\mathchardef\za="710B  %\alpha
\mathchardef\zb="710C  %\beta
\mathchardef\zg="710D  %\gamma
\mathchardef\zd="710E  %\delta
\mathchardef\zve="710F %\epsilon
\mathchardef\zz="7110  %\zeta
\mathchardef\zh="7111  %\eta
\mathchardef\zvy="7112 %\theta
\mathchardef\zi="7113  %\iota
\mathchardef\zk="7114  %\kappa
\mathchardef\zl="7115  %\lambda
\mathchardef\zm="7116  %\mu
\mathchardef\zn="7117  %\nu
\mathchardef\zx="7118  %\xi
\mathchardef\zp="7119  %\pi
\mathchardef\zr="711A  %\rho
\mathchardef\zs="711B  %\sigma
\mathchardef\zt="711C  %\tau
\mathchardef\zu="711D  %\upsilon
\mathchardef\zvf="711E %\phi
\mathchardef\zq="711F  %\chi
\mathchardef\zc="7120  %\psi
\mathchardef\zw="7121  %\omega
\mathchardef\ze="7122  %\varepsilon
\mathchardef\zy="7123  %\vartheta
\mathchardef\zf="7124  %\varomega
\mathchardef\zvr="7125 %\varrho
\mathchardef\zvs="7126 %\varsigma
\mathchardef\zf="7127  %\varphi
\mathchardef\zG="7000  %\Gamma
\mathchardef\zD="7001  %\Delta
\mathchardef\zY="7002  %\Theta
\mathchardef\zL="7003  %\Lambda
\mathchardef\zX="7004  %\Xi
\mathchardef\zP="7005  %\Pi
\mathchardef\zS="7006  %\Sigma
\mathchardef\zU="7007  %\Upsilon
\mathchardef\zF="7008  %\Phi
\mathchardef\zW="700A  %\Omega

\newcommand{\be}{\begin{equation}}
\newcommand{\ee}{\end{equation}}

\newcommand{\bea}{\begin{eqnarray}}
\newcommand{\eea}{\end{eqnarray}}
\newcommand{\beas}{\begin{eqnarray*}}
\newcommand{\eeas}{\end{eqnarray*}}
\def\*{{\textstyle *}}
\newcommand{\pa}{\partial}
\newcommand{\ti}{\times}

\newcommand{\Ll}{{\pounds}}

\def\bN{{\mathbf N}}

\def\sT{{\mathsf T}}

\def\xd{\mathrm{d}}

\def\dt{\xd_{\mathsf T}}

\newdir{|>}{%
!/4.5pt/@{|}*:(1,-.2)@^{>}*:(1,+.2)@_{>}}

\newcommand{\Ci}{C^{\infty}}

\newcommand{\N}{\mathbb{N}}
\newcommand{\Z}{\mathbb{Z}}
\newcommand{\R}{\mathbb{R}}

\newcommand{\ao}{\mathbf{1}}

\newcommand{\n}{\nabla}

\newcommand{\id}{\on{id}}

\newcommand{\catname}[1]{\textnormal{\texttt{#1}}}
\newcommand{\cn}[1]{\textnormal{\texttt{#1}}}

\newcommand{\mn}{{\medskip\noindent}}

\newcommand{\no}{{\noindent}}

\def\on{\operatorname}

\newcommand{\w}{{\mathsf w}}

\newcommand{\Sgroup}{\mathbb{S}}
\newcommand{\Sn}{{\mathbb{S}_n}}
\DeclareMathOperator{\GL}{GL}

\def\bl{\big(}
\def\br{\big)}
\def\Bl{\Big(}
\def\Br{\Big)}
\def\sgn{\on{sgn}}
\def\nb{ \catname{WghtB}}
\def\nbn{ \catname{WghtB(n)}}
\def\nbbn{ \catname{WghtB}^\ze(n)}
\def\nm{ \catname{WghtB}^\ze}

\def\ssln{ \catname{SkewVB}_0[n]}
\def\ssvbn{\catname{SkewVB}[n]}
\def\vbn{ \catname{VB}[n]}
\def\vbne{ \catname{VB}^\ze[n]}
\def\vb{ \catname{VB}}
\def\cln{ \catname{VB}_0[n]}
\def\svbn{ \catname{SymVB}[n]}
\def\svbne{ \catname{SymVB}^\ze[n]}

\def\z2{{\Z_2^n}}

\newcommand{\az}{{z^A_\pi}}%{{\bar z}}

\newcommand{\es}{{E^\zs}}
\def\cP{{\mathcal P}}
\def\fL{{\mathfrak L}}

\begin{document}
\title{\textbf{Equivalence functors\\ in graded supergeometry}\footnote{This research was partially funded by the National Science Centre (Poland) within the project WEAVE-UNISONO, No. 2023/05/Y/ST1/00043.}}
\date{}
\author{\\ Katarzyna  Grabowska$^1$\\ Janusz Grabowski$^2$\\ Miko\l aj Rotkiewicz$^3$
        \\ \\
         $^1$ \textit{Faculty of Physics,\
                University of Warsaw}\\
                $^2$ \textit{Institute of Mathematics,\
                Polish Academy of Sciences}\\
                $^3$ \textit{Institute of Mathematics,\
                University of Warsaw}
                }
%\author{Janusz Grabowski}
\maketitle

\begin{abstract}
It has recently been proved that the category of $\bN$-manifolds of degree $n$, that is, $\N$-graded supermanifolds of degree $n$ for which the parity agrees with the gradation, is equivalent to the category of purely even $n$-tuple vector superbundles equipped with a suitable action of the symmetric group $\Sn$ permuting the vector bundle structures. This equivalence may be interpreted as a \emph{desuperization} of $\bN$-manifolds.
In the present paper, we place this result within a broader framework of graded structures on supermanifolds and explicitly describe several canonical equivalences between the corresponding categories in a purely geometric, constructive manner. The desuperization equivalence functor appears as a composition of some of these canonical equivalences. Our constructions are entirely canonical and rely on standard tools of supergeometry, including iterated tangent functors, parity reversion in vector superbundles, and the interpretation of $n$-tuple vector bundles in terms of commuting Euler vector fields associated with the underlying vector bundle structures.

\par

\end{abstract}

\section{Introduction}
The term `graded manifold' appears in the literature with several distinct meanings, including those referring solely to a $\mathbb{Z}_2$-grading (parity), that is, to \emph{supermanifolds} (see, e.g.,~\cite{Kostant:1977}).
Graded supermanifolds constitute fundamental objects in many areas of geometry and mathematical physics, including the BV and AKSZ formalisms of quantum field theory, SUSY field theory, BRST theory, and supergravity.
Although supermanifolds are often referred to simply as \emph{graded manifolds}, we shall consistently use the term `supermanifold', reserving `graded manifold' for additional gradings, such as $\N^n$- or $\z2$-gradings. The graded structure is typically encoded in the structure sheaf of a ringed space, which is assumed to be a sheaf of graded (super)algebras. For an approach to graded (super)manifolds allowing arbitrary real weights, we refer to \cite{Grabowska:2024}. It is based on a concept of homogeneity associated with a \emph{weight vector field} and allows for homogeneous Poincar\'e Lemma, homogeneous Frobenius Theorem, and homogeneous Darboux theorems (for the latter see also \cite{Grabowski:2026}).

\mn Within the category of graded manifolds, non-negatively graded supermanifolds play a particularly important role. When the Grassmann parity of the coordinates coincides with the parity of their weight, one obtains an \emph{$\bN$-manifold}, a notion introduced in 2001 by \v{S}evera (see~\cite{Roytenberg:2002,Severa:2005}). Together with its variants, this class forms a central object of graded supergeometry and its applications (cf.~\cite{Aleksandrov:1997}), including $dg$-, $NQ$-, and $NP$-manifolds. A prominent example is provided by \emph{Courant algebroids}, originally introduced as a rather involved structure in purely even differential geometry \cite{Liu:1997}, and subsequently identified by Roytenberg \cite{Roytenberg:2002} as Hamiltonian systems on symplectic $\bN$-manifolds of degree~2.

\mn In the purely even setting, Grabowski and Rotkiewicz \cite{Grabowski:2012} introduced the notion of \emph{graded bundles}, a distinguished class of non-negatively graded manifolds. They showed that a graded bundle structure is equivalent to the existence of a smooth action of the multiplicative monoid $(\R,\cdot)$ on the underlying manifold. Such actions were termed \emph{homogeneity structures} and studied further in \cite{Grabowska:2021}. This observation was subsequently extended to supermanifolds by J\'o\'zwikowski and Rotkiewicz \cite{Jozwikowski:2016}.

\mn It was observed in \cite{Bruce:2016} that graded bundles admit a canonical \emph{linearization} procedure, which lowers the degree of the grading at the expense of introducing an additional vector bundle structure. This idea was applied to higher-order mechanics in \cite{Bruce:2015a}. Iteration of the linearization leads to a \emph{polarization} procedure producing \emph{symmetric $n$-tuple vector bundles} from graded bundles of degree~$n$, yielding an equivalence of categories \cite{Bruce:2016a}. Following the approach of \cite{Grabowski:2009}, an \emph{$n$-tuple vector bundle} (or \emph{$n$-vector bundle}) can be understood as a manifold equipped with $n$ vector bundle structures whose Euler vector fields commute, both in the purely even and in the supergeometric setting (cf.~\cite{Jozwikowski:2016}).

The term `symmetric' refers to the presence of an additional action of the symmetric group $\Sn$ permuting the Euler vector fields. A canonical example is provided by the iterated tangent bundle $\sT^{(n)}M=\sT\sT\cdots\sT M$, on which $\Sn$ acts by compositions of diffeomorphisms induced by the canonical flips $\zk_N:\sT\sT N\to\sT\sT N$ (see \cite{Bruce:2016a,Pradines:1974}).

Geometrically, polarization may be viewed as a procedure applied to admissible polynomial changes of local coordinates. The resulting coordinates arise from repeated differentiation, corresponding to the iterated application of the tangent functor. This leads to an enlarged manifold endowed with an $n$-fold vector bundle structure, from which the original graded bundle can be recovered as the submanifold of \emph{holonomic vectors}. This construction parallels the classical embedding of higher tangent bundles $\sT^nM$ into iterated tangent bundles $\sT^{(n)}M$ as the submanifold of holonomic vectors. The terminology `polarization' reflects its analogy with the polarization of homogeneous polynomials into multilinear maps.

\mn In the supergeometric setting, a related equivalence of categories was first established by Jotz \cite{JotzLean:2015}, who showed that $\bN$-manifolds of degree~2 are categorically equivalent to what she termed \emph{metric double vector bundles}. Her approach differs substantially from that of \cite{Bruce:2016a}. Closely related results were obtained independently in the PhD thesis of del~Carpio--Marek \cite{Carpio-Marek:2015}, formulated in terms of double vector bundles equipped with an involution. More recently, the result by Jotz was extended to arbitrary degrees by Heuer and Jotz \cite{Heuer:2024}; see also the related work of Vishnyakova \cite{Vishnyakova:2019,Vishnyakova:2024}.

The equivalence functor introduced in \cite{Heuer:2024}, termed the \emph{geometrization of an $\bN$-manifold}, establishes an equivalence between the category $\nbbn$ of $\bN$-manifolds of degree~$n$ and the category of purely even $n$-vector bundles, referred to there as \emph{symmetric}. We prefer the term \emph{skew-symmetric}, as transpositions in $\Sn$ are accompanied by sign changes, reserving the term \emph{symmetric} for the genuinely symmetric case studied in \cite{Bruce:2016a}. Accordingly, we refer to the procedure as \emph{desuperization}. Note also a recent paper on a functor between two categories of $\Z$-graded manifolds \cite{Guarin:2026}.

\mn the present paper aims to provide a concise, explicit, and purely geometric construction of the desuperization based on canonical tools of supergeometry. We work in a more general setting than that of $\bN$-manifolds, considering manifolds equipped with an $\N$-gradation not necessarily inducing the parity, referred to as \emph{$\N$-weighted bundles}, which may be regarded as super-analogues of graded bundles in the sense of Grabowski and Rotkiewicz. Since the existence of an equivalence alone is of limited value without an explicit construction, our functors are defined constructively. A preliminary outline of this idea already appeared in \cite{Bruce:2016a}, where polarization was combined with parity reversion.

\mn { We summarize our results, see Theorems~\ref{main}, \ref{m1}, and \ref{m2},  in the following diagram: }

\[
\xymatrix@C=7.5em@R=5.5em{
\cn{WghtB}(n)
  \ar@<0.6ex>[r]^{\mathcal{P}^{(n)}\ \text{(polarization)}}
  \ar@/^2pc/[rr]^{\text{reverse polarization}}
&
\cn{SymVB}[n]
  \ar@<0.6ex>[l]^{\text{diagonalization}}
  \ar[r]^{\Xi}
&
\cn{SkewVB}[n]
\\
\cn{WghtB}^\ze(n)
  \ar@{^{(}->}[u]
  \ar[r]^{\mathcal{P}^{(n)}}
  \ar@/_2pc/[rr]_{\text{desuperization}}
&
\cn{SymVB}^\ze[n]
  \ar@{^{(}->}[u]
  \ar[r]^{\Xi}
&
\cn{SkewVB}_0[n]
  \ar@{^{(}->}[u]
}
\]

\mn{
(We use a unified notation for the categories mentioned above.
The symbol $\cn{VB}$ denotes the category of vector bundles
(which by default are understood as super vector bundles),
$\catname{WghtB}$ the category of $\N$-weighted bundles,
and $\cn{VB}[n]$ the category of $n$-vector bundles.
The subscript $0$ indicates purely even objects,
while $\ze$ signifies that the parity is given by the total weight modulo~2. Moreover, the parentheses $(n)$ indicate the degree. In particular, $\cn{WghtB}^\ze(n)$ coincides with the category of $\bN$-manifolds of degree $n$.)

In short, the equivalence functor $\Xi$ between the categories of symmetric
and skew-symmetric $n$-vector superbundles arise as the restriction of the
total parity reversion functor, a result interesting \emph{per se}.
Our desuperization functor is obtained by composing the polarization functor,
restricted to the category of $\bN$-manifolds, with the functor~$\Xi$.
The result of Heuer and Jotz~\cite{Heuer:2024} is recovered by restricting
the reverse polarization to $\bN$-manifolds.}

\mn The paper is organized as follows. In Section 2, we introduce the concept of an $\N$-weighted bundle of degree $n$, being a sort of an $\N$-graded superbundle, \emph{via} the concept of a \emph{weight vector field}. Vector superbundles appear to be just $\N$-weighted bundles of degree 1. Then, we view $\bN$-manifolds as particular $\N$-weighted bundles, where weights determine the parity. Section 3 is devoted to analyzing the structure of $n$-vector bundles. In \cite{Heuer:2024}, $n$-vector bundles are understood as a certain functor; we prefer to use a more intuitive and geometric approach \emph{via} a sequence of commuting Euler vector fields. The \emph{total parity reversion functor} for $n$-vector bundles is discussed in Section 4. We get a generalization of the result of Voronov \cite{Voronov:2012} in this context, and a skew-symmetric $\Sn$-action appears.

{
In Section~5, we study the canonical example of the iterated tangent bundle $\sT^{(n)}M$, where the $\Sn$-action induced by the canonical flips is \emph{symmetric} in the sense that it acts trivially on the cores of all double vector bundle structures and permutes the corresponding Euler vector fields. This serves as our canonical model of a \emph{symmetric $n$-vector bundle}. In Section~6, we apply the total parity reversion to this model, thereby introducing  the notion of a skew-symmetric $n$-vector bundle.

Lemma~\ref{l:adapted_coord} guarantees the existence of a convenient system of adapted local coordinates on any symmetric $n$-vector bundle. This result is important in the proof of the canonical equivalence  between symmetric and skew-symmetric $n$-vector bundles.
}
Finally, in Section~7, we introduce the \emph{polarization functor}, a supergeometric generalization of the polarization construction of \cite{Bruce:2016a}. Composed with the equivalence between symmetric and skew-symmetric $n$-vector bundles, it yields the desuperization equivalence functor.

\section{Weight vector fields, homogeneity, and $\bN$-manifolds}

Let us consider the superspace $\R^{p|q}$ with standard global coordinates $(z^A)=(\chi^I,\zx^i)$, where $\chi^I$ are even and $\zx^i$ are odd, and associate with each $z^A$ a positive integer $w_A$. This can be encoded in the \emph{weight vector field}
$$\n_0=\sum_Aw_Az^A\pa_{z^A},$$
or the \emph{dilation} $h^0:\R\ti\R^{p|q}\to\R^{p|q}$,
$$h^0_t(z^A)=\bl t^{w_A}z^A\br.$$
It is easy to see that $h^0$ is a smooth action of the monoid $(\R,\cdot)$ of multiplicative reals,
\be\label{monoid} h^0_1=\id,\quad h^0_t\circ h_s^0=h_{ts}^0.\ee

\mn The weight vector field determines a concept of \emph{homogeneity} on the supermanifold $\R^{p|q}$: a tensor field $K$ is \emph{homogeneous with weight (degree) $w$} if $\Ll_{\n_0} K=w\cdot K$. Here, $\Ll$ denotes the Lie derivative. We will write $\w(K)=w$. In particular, a function $f$ is homogeneous of {weight } $w$ if $\n_0(f)=wf$ (equivalently, $f\circ h^0_t=t^wf$), and a vector field $X$ is homogeneous of {weight } $w$ if $[\n_0,X]=wX$. Of course, the concept of homogeneity is consistent with the original {weights, } as the coordinate $z^A$ is homogeneous of {weight } $w_A$. It is easy to see that for homogeneous functions $f,g$ their product $fg$ is homogeneous of {weight } $\w(f)+\w(g)$, and every homogeneous function is a polynomial in variables $(z^A)$ (cf. \cite[Lemma 2.1]{Grabowski:2012}).
 In particular, the {weights } of homogeneity of functions on $(\R^{p|q},\n_0)$ are non-negative integers.
 %If all weights satisfy
 Structures $(\R^{p|q},\n_0)$ we will call \emph{$\N$-weighted (super)spaces.} If all weights satisfy $w_A\leq n$, we say that the space is of degree $n$. In particular, a space of degree $n$ is also of degree $n'$ for any $n'>n$.

\mn It is easy to see that the weight vector field is even and complete with the flow $\zf^s$, and induces the dilation, $h^0_{e^s}=\zf^s$. If we view $\R^{p|q}$ as a supermanifold, then morphisms $\zf:\R^{p|q}\to\R^{p|q}$
respecting the dilation, $\zf\circ h^0_t=h^0_t\circ\zf$ are exactly those relating the weight vector field in the sense that the diagram (remember that $\n_0$ is even, so $\n_0:\R^{p|q}\to\sT\R^{p|q}$ is a morphism of supermanifolds)
\[
\xymatrix@C+30pt@R+15pt{
\sT \R^{p|q} \ar[r]^{\sT\zf} & \sT \R^{p|q} \\
 \R^{p|q} \ar[u]^{\n_0}\ar[r]^{\zf} & \R^{p|q}\ar[u]_{\n_0}}
\]
is commutative.

\mn It is natural to extend the concept of a weighted space to that of an \emph{$\N$-weighted bundle}.

\begin{definition}
An \emph{$\N$-weighted bundle of degree $n$} is a fiber (super)bundle $\zt:F\to M$ {over a supermanifold $M$ } modelled on a weighted (super)space $(\R^{p|q},\n_0)$ of degree $n$. We require that the principal bundle $F$ is equipped with an atlas of local trivializations such that the transition maps respect the weight vector field in fibers (equivalently, they commute with the dilations $h_t$ in fibers). An $\N$-weighted bundle is called an $\bN$-manifold if $M$ is purely even and the parity of coordinates in $(\R^{p|q},\n_0)$ agrees with their weights: the weights of even coordinates are even, and the weights of odd coordinates are odd.
\end{definition}
Note that in this case, $M$ is canonically embedded in $F$ as the submanifold of zeros in fibers.  It follows that we have local trivializations $\zt^{-1}(U)\simeq U\ti\R^{p|q}$, {where $U$ ranges over  an open cover of the supermanifold $M$,} with coordinates $(x^a)=(u^s,z^A)$, and a globally defined weight vector field $\n$ which is vertical and reduces to $\n_0$ in fibers. In coordinates,
$$\n=\sum_Aw_Az^A\pa_{z^A},$$
which means that the only difference with the weighted space is the existence of base coordinates $u^s$, which are homogeneous of {weight } 0. Similarly, fiber dilations $h^0_t$ give rise to a monoid $(\R,\cdot)$-action $h:\R\ti F\to F$ on $F$ which reduces to $h^0$ in fibers. In particular, $\zt=h_0$ and transition maps commute with $h_t$.
Moreover, the generator of the group action $\R_+\ti E\to E$ is the weight vector field $\n=D_th: E\to\sT E$.
\emph{Morphisms} in the category $\nb$ of $\N$-weighted bundles {are defined as} smooth maps (morphisms of supermanifolds) relating the corresponding weight vector fields,
\be\label{mor}
\xymatrix@C+30pt@R+15pt{
\sT F^1 \ar[r]^{\sT\zf} & \sT F^2 \\
 F^1 \ar[u]^{\n^1}\ar[r]^{\zf} & F^2\ar[u]_{\n^2}.}
\ee
We will write it as $\zf_*(\n^1)\subset\n^2$. {Since we can reduce ourselves to the typical fiber, it is easy to see that this can be restated as $\zf:F^1\to F^2$ intertwines $h^1$ and $h^2$,
\be\label{mor1} \zf\circ h^1_t=h^2_t\circ\zf.
\ee}
Clearly, $\N$-weighted bundle of degree { $n$} form a full subcategory $\nbn$ in $\nb$ as well as $\bN$-manifolds of degree $n$ form a full subcategory $\nbbn$ in $\nm$.

\mn Again, homogeneous functions are locally polynomials in homogeneous fiber coordinates $z^A$ with coefficients being this time functions $f(u)$ on the base manifold $M$. Hence, the transition functions take the form
\be\label{transitions}
(u')^s=(u')^s(u),\quad (z')^{A} = \sum_{w_A=w_{A_1}+\cdots +w_{A_k}} z^{A_{1}} z^{A_{2}} \cdots z^{A_{k}}\cdot T^A_{A_{k} \cdots A_{2} A_{1}}(u).
\ee
Of course, $T^A_{A_{n} \cdots A_{2} A_{1}}(u)$ is totally symmetric in lower indices in the purely even case, and changes the sign appropriately to the sign obtained from permutations of supercoordinates in $z^{A_{1}} z^{A_{2}} \cdots z^{A_{n}}$. It is also clear that, dealing with a fiber bundle modelled on the weighted superspace $(\R^{p|q},\n_0)$ with an atlas of local trivializations for which the transition maps take the form (\ref{transitions}), we actually work with an $\N$-weighted bundle.
\begin{remark}
Our $\N$-weighted bundles in the purely even case have been called in \cite{Grabowski:2012} \emph{graded bundles} and smooth actions $h$ of the monoid $(\R,\cdot)$ -- \emph{homogeneity structures}. The main result there states that any homogeneity structure is associated with a graded bundle as above. This has been extended to supermanifolds by J\'o\'zwikowski and Rotkiewicz \cite{Jozwikowski:2016}. Note that $\bN$-manifolds were introduced by Roytenberg \cite{Roytenberg:2002} and {\v{S}}evera \cite{Severa:2005}. Actually, {\v{S}}evera defined it as a homogeneity structure for which $h_{-1}$ acts as the parity map, although without any proof that we get a weighted structure. An approach to $\bN$-manifolds \emph{via} coalgebra bundles appeared recently in \cite{Bursztyn:2025}. Note also a definition of a $\Z$-graded manifold by Th.~Voronov \cite{Voronov:2002} which is analogous to our $\N$-weighted bundles but with weights being arbitrary integers.
\end{remark}
We define \emph{$n$-fold $\N$-weighted bundles} as supermanifolds equipped with a sequence of $n$ commuting weight vector fields $\bl\n^1,\dots,\n^n\br$. It is easy to see that their sum, $\n=\n^1+\cdots+\n^n$, is also a weight vector field, called the \emph{total weight vector field}. Indeed, $\n$ is the generator of the homogeneity structure $h_t=h^1_t\circ\cdots\circ h^n_t$, where $h^i$, $i=1,\dots,n$, are the (commuting!) homogeneity structures associated with the weight vector fields. Moreover, one can prove that we can always find local coordinates that are simultaneously homogeneous with respect to all the weight vector fields. {The proof in the super case is analogous to that in the purely even case \cite{Grabowski:2012}.}

\section{Multiple vector bundles}

This section is devoted to $n$-fold vector superbundles  $(E; \n^1, \ldots, \n^n)$ in which each weight vector field $\n^k$ defines a super vector bundle on the supermanifold $E$.

In the standard (even) differential geometry, it is well known \cite{Grabowski:2009} that a vector bundle (VB) $\zt:E\to M$ is completely determined by its multiplication $h:\R\ti E\to E$ by reals or, equivalently, by the corresponding Euler vector field
$$\n_E(v)=\frac{\xd}{\xd t}\,\Big|_{t=1}h_t(v).$$
In local affine coordinates $(x^a)=(x^a,y^i)$, the Euler vector field reads
\be\label{El} \n_E=y^i\pa_{y^i}.\ee
 It extends easily to the case of vector superbundles, so \textbf{we will skip the suffix `super' in the rest of the paper.} The form (\ref{El}) of $\n_E$ means that vector bundles are nothing but $\N$-weighted bundles of degree 1, and the Euler vector field is a particular weight vector field.  Consequently, the category of vector bundles is in this way a full subcategory in the category $\nb$ of $\N$-weighted bundles.

Homogeneous functions of {weight } 1 are linear in the obvious sense (cf. the Euler's Homogeneous Function Theorem), and we call a vector field $X$ on $E$ \emph{linear} if it is homogeneous of {weight } 0, i.e., it commutes with $\n_E$, $[\n_E,X]=0$. Affine coordinates $(u^s,y^A)$ on vector bundles have weights 0 (basic) or 1 (linear coordinates in fibers). Moreover, $\zt=h_0$ is a smooth projection onto the base manifold $M=h_0(F)$ which is therefore canonically embedded in $E$. It can also be described as the set of zeros of the Euler vector field. Since the multiplication by reals in affine coordinates reads
$$h_t\bl u^s,z^A\br=(u^s,tz^A),$$
it is easy to see that the homogeneity structure $h$ is in this case \emph{regular}, i.e., for all $t\ne 0$ the derivative $\sT h_t$ at points of $M$ is injective on the kernel of $\sT h_0$. In other words,
$$\frac{\pa h}{\pa t}(0,\cdot):E\to\sT E$$
takes the value 0 exactly on $M$. An interesting observation in this context is the following characterization of homogeneity structures of vector bundles.

\begin{theorem}[\cite{Grabowski:2009}] A homogeneity structure $h$ on a manifold $F$ is regular if and only if it is a multiplication by reals for a vector bundle structure on $h_0:F\to M=h_0({ F})$.
\end{theorem}
\no The above observation substantially simplifies working with a vector bundle. In particular, VB-morphisms are just smooth maps between the total spaces intertwining the multiplications by reals, and vector subbundles are just submanifolds invariant with respect to the multiplication by reals. Equivalently, VB-morphisms are smooth maps relating the Euler vector fields, and vector subbundles are submanifolds $E_0$ to which the Euler vector field is tangent and its restriction is complete, i.e., $\n\,\big|_{E_0}$ is an Euler vector field on $E_0$. It is now obvious that the flows of linear vector fields consist of VB-morphisms. Note that our definition does not require any use of the Lie supergroups $\GL(p|q)$ and that this setting works well also for analogs of supermanifolds with other gradings and sign rules, e.g., $\z2$-supermanifolds \cite{Bruce:2024,Covolo:2014a,Covolo:2014b}.

\begin{example}
A canonical vector bundle associated with a supermanifold $M$ of dimension $(p|q)$ is its tangent bundle $\zt_M:\sT M\to M$. If $(U,x^a)$ is a coordinate chart in $M$ then $\zt_M^{-1}(U)\simeq U\ti\R^{p|q}$ is a coordinate chart in $\sT M$ with coordinates $(x^a,\dot x^b)$, where the parity of the linear coordinates $\dot x^a$ in $\R^{p|q}$ is the same as that of $x^a$. The transition map in $\sT M$ associated with a change of coordinates $y^a=y^a(x)$ in $M$ is the differential of the latter,
$$(x')^a=(x')^a(x),\quad (\dot x')^a=\dot x^b\frac{\pa (x')^a}{\pa x^b}(x).$$
Note that it is essential to keep the coefficients  $\frac{\pa (x')^a}{\pa x^b}$ to the right of the linear coordinates $\dot x^b$.
These transition maps respect the parity and the Euler vector field
$$\n_M=\dot x^a\pa_{\dot x^a},$$
so we get a vector { bundle } denoted $\sT M$. The coordinates $(x^a,\dot x^b)$ we call \emph{adapted} for $(x^a)$. Note that, analogously to the purely even case, $\sT M$ can also be obtained as the first jet bundle of `curves' in $M$ \cite{Bruce:2014}.
\end{example}

\mn Given an atlas of affine coordinates $(u^s,z^A)$, transition maps preserve the homogeneity {weight}, so can be written as
\be\label{vbtf}
(u')^s=(u')^s(u),\quad (z')^A=z^B\cdot F^A_B(u).
\ee
For vector superbundles $\zt:E\to M$, we will fix the convention that we write linear functions first and the coefficients, being basic functions, second. This form of transition maps immediately implies that they remain valid if we reverse the parity of the linear coordinates $(z^A)$. In this way, we get a new vector bundle
$\zP\zt:\zP E\to M$. It is obvious that if $\zf:E^1\to E^2$ is a VB-morphism, then the map $\zf$, having formally the same form (\ref{vbtf}) in affine coordinates with reversed parities in fibers, yields a VB-morphism $\zP\zf:\zP E^1\to\zP E^2$.
\begin{proposition}
The \emph{parity reversion} $\zP$ is an equivalence functor from the category of vector superbundles $\vb$ into the same category.
\end{proposition}
%\begin{remark} As we will work mostly with supermanifolds and vector superbundles, we will generally skip the %prefix `super', indicating separately the fact that we work with a purely even case.
%\end{remark}
\no Our framework extends almost trivially to the case of multiple vector bundles.
\begin{definition}
An \emph{$n$-fold vector bundle} (\emph{$n$-vector bundle}, in short) is a supermanifold $E$ equipped with an ordered set of $n$ vector bundle structures $(\zt_1,\dots\zt_n)$, where $\zt_i:E\to M_i$, whose homogeneity structures (multiplications by reals) $h^i$ pairwise commute,
$$h^i_t\circ h^j_s=h^j_s\circ h^i_t\quad\text{for all}\quad t,s\in\R\ \text{and}\  i,j=1,\dots,n\,.$$
Equivalently, whose Euler vector fields $\n^i$ pairwise commute
$$ [\n^i,\n^j]=0, \quad \text{for all}\quad i,j=1,\dots,n\,.$$
\end{definition}
\no The equivalence of the above conditions comes from the fact that $h^j$ is the unique smooth extension of the one-parameter group of diffeomorphisms $(h_t^j)_{t>0}$ whose generator is $\n^j$.

\mn A fundamental observation is the following (cf. \cite{Grabowski:2009}).
\begin{proposition}
If $E$ is an $n$-vector bundle, then there is an atlas of coordinate charts on $E$ whose coordinates are $n$-homogeneous with {$n$-weights } in $\{0,1\}^n$. Moreover, the tangent $\sT E$ and the cotangent bundle $\sT^*E$ are canonically $(n+1)$-vector bundles.
\end{proposition}

\no Such coordinates we will call \emph{(multi)affine}. Like in the case of a single vector bundle, it is convenient to distinguish homogeneous coordinates $(u^s)$ of {weight } $0^n\in\zone$ (base coordinates) from the rest of the affine coordinates $(z^A)$ with non-zero {weights, }
$$w_A=\bl w^1_A,\dots,w^n_A\br\in\zone\setminus\{0^n\}.$$
The base coordinates are coordinates on the \emph{total base} $M=\bigcap_{i}M_i$.
The transition functions look exactly like in (\ref{transitions}) with the only difference that $w_A\in\zone\setminus\{0^n\}$. This implies automatically that any coordinate $z^A$ appears in the products $z^{A_{1}} z^{A_{2}} \cdots z^{A_{k}}$ not more than once, and shows that our $n$-vector bundles coincide with the $n$-tuple vector bundles described by T.~Voronov \cite{Voronov:2012} and (for $n=3$) the triple vector bundles studied by K.~Mackenzie \cite{Mackenzie:2005}.

\mn It is clear that $n$-vector bundles form a category $\vbn$, in which morphisms between
\newline $\bl E,\n^1,\dots,\n^n\br$ and $\bl E',(\n')^1,\dots,(\n')^n\br$ are smooth maps $\zf:E\to E'$  which are respecting the collective homogeneity, i.e., they are VB-morphisms for every pair $\bl\zt^i,(\zt')^i\br$. We will also consider purely even (traditional) $n$-vector bundles, which form a full subcategory $\cln$ in $\vbn$.

\mn Note that in the above definition, the order of the vector bundle structures does matter. We can permute the position of the Euler vector fields using any permutation $\zs$ from the symmetric group $\Sn$ understood as the group of bijections of the set $\{1,\dots,n\}$:
the permuted Euler vector fields are ordered as $\bl E,\n^{\zs(1)},\dots,\n^{\zs(n)}\br$, which we denote shortly $\es$ (cf.\cite{Bruce:2016a}). In particular, the {weight } of $z^A$ in $\es$ is
$$
{w}_A^\zs=\bl w_A^{\zs(1)},\dots,w_A^{\zs(n)}\br
$$
and, to indicate this fact, the coordinate $z^A$ in $\es$ we will denote ${z}^A_\zs$. As the induced action of $\Sn$ on sequences is a right action, we have
\be\label{wzs1}z^A_{\zs'\zs}=(z^A_{\zs'})_{\zs}\quad\text{and}\quad w_A^{\zs'\zs}=\bl w_A^{\zs'}\br^\zs.\ee
Actually, for any $\zs\in\Sn$, this defines an equivalence \emph{permutation functor}
\be\label{Pzs}P^\zs:\vbn\to\vbn,
\ee
which associates $\es$ to $E$ such that
\be\label{Pcomp}P^{\zs'\zs}=P^{\zs}\circ P^{\zs'}, \quad {E}^{\zs'\zs}={\bl E^{\zs'}\br}^{\zs}.
\ee
Of course, $E$ and $\es$ are the same supermanifolds, only the $n$-vector bundle structures are permuted. Moreover, with any morphism of $n$-vector bundles $\zf:E_1\to E_2$ there is associated the whole family of $n$-vector bundle morphisms $P^{\zs}(\zf):E_1^\zs\to E_2^\zs$, enumerated by $\zs\in\Sn$, since $\zf_*(\n^i)\subset(\n')^i$ for all $i$ is the same condition as $\zf_*(\n^{\zs(i)})\subset(\n')^{\zs(i)}$ for all $i$.
\begin{remark}\label{zs0} All $n$-vector bundle morphisms $P^{\zs}(\zf):E_1^\zs\to E_2^\zs$ refer to the same morphism $\zf:E_1\to E_2$ of supermanifolds, only the ordering of vector bundle structures changes.
If the context is clear, we will often skip the part $P^\zs$ and write simply $\zf$, understanding this as the above family of $n$-vector bundle morphisms.
\end{remark}

%%%%%%%%%%%%%%%%%%%%%%%%%%%%%%%%%%%%%%%%%

\mn Like $\bN$-manifolds form a subcategory of $\N$-weighted bundles, in the category $\vbn$ of $n$-vector bundles, we can also distinguish its subcategory, relating {multi-weights} to the parity.
\begin{definition}
An $n$-vector bundle in which affine coordinates of weight $\za\in\zone$ have the parity $|\za|=\sum_i\za^i$ we call an \emph{$[n]$-vector bundle}. { The corresponding category is denoted by  $\catname{VB}^\ze[n]$.} %\commentMRR{I chyba dopisać $\catname{VB}^\ze[n]$.}
\end{definition}
\begin{remark}
The concept of a \emph{double vector bundle} (DVB) (i.e., a $2$-vector bundle) goes back to Pradines \cite{Pradines:1974} (see also \cite{Chen:2014,Konieczna:1999}). It is sometimes characterized as a `vector bundle in the category of vector bundles', but an extension to $n$-vector bundles is rather uneasy in this framework. In \cite{Heuer:2020,Heuer:2024}, $n$-vector bundles are described as a certain functor. This is elegant, but does not seem to be very operational for applications.
\end{remark}
\no Given an $n$-vector bundle $E$ and $\za\in\{ 0,1\}^n$, the submanifold
$$E_\za=\bigcap_{\za^i=0}M_i$$ is itself an $\vert \za\vert$-vector bundle with respect to the (restricted) Euler vector fields $\{\n^k:\za^k=1\}$ determining vector bundles
$$\zt_\za^k:E_\za\to E_{\za'_k},$$
where $\za'_k=(\za^1,\dots,\za^{k-1},0,\za^{k+1},\dots,\za^N)$.
Moreover, the projections $\zt_\za^k$ are morphisms of multivector bundles. In this way, we get the \emph{characteristic diagram} of the $n$-vector bundle $E$,
which is a commutative diagram of multivector bundles with $2^n$ vertices $E_\za$ and vector bundle morphisms $\zt_\za^k$. For instance, the characteristic diagram for a triple vector bundle (cf. \cite{Grabowski:2009,Mackenzie:2005}) looks like

\begin{equation}\label{triple} \xymatrix@C-3pt@R+5pt{
               &  M_1=E_{011}\ar[dd]\ar[dl]   &           &  E=E_{111} \ar[ll]_{\zt_1} \ar[dd]^{\zt_3} \ar[dl]_{\zt_2}
  \\
E_{001}\ar[dd] &            & M_2=E_{101}\ar[ll]\ar[dd] &
  \\
               &  E_{010} \ar[ld] &                 &  M_3=E_{110} \ar[ll] \ar[ld]
  \\
M=E_{000}        &                   &      E_{100}\ar[ll] & }
\end{equation}

\mn\begin{example}\label{triv}
For a manifold $M$ and $\za\in\zone$, $\za\ne 0$, let $V(\za)$ be a vector bundle over $M$ and $\n^\za$ its Euler vector field. Then, each manifold
$$
E_\za={ \prod_{\zb\le\za}{}_{{}_M}V(\zb)}$$ is canonically a vector bundle over $E_{\za'_k}$ if $\za_k=1$, and the corresponding Euler vector field is
$$\n^\za_k=\sum_{\zb\le\za,\ \zb_k=0}\n^\zb,$$
so we get the characteristic diagram for the $n$-vector bundle $E=E_\ao$ with $n$ vector bundle structures
$\zt_i:E\to M_i=E_{\ao'_i}$. Such $n$-vector bundles we call \emph{trivial}.
\end{example}
\no An important observation is that any $n$-vector bundle $E$ is locally (over some open $U\subset M$) trivial. Indeed, for affine coordinates $(u^s,z^A)$ on $E$ over $U$ and $\za\in\zone$, $\za\ne 0$, we define locally $E_U[\za]$ as the submanifold
\be \label{df:E[alpha]}
E_U[\za]=\big\{(u^s,z^A)\,|\, z^A=0\ \text{if}\  \w(z^A)\ne\za\big\}.
\ee

It is a trivial vector bundle over $U$.
It is easy to see that the transition maps,
\be\label{trans}(z')^{A} = \sum_{w_A=w_{A_1}+\cdots +w_{A_k}} z^{A_{1}} z^{A_{2}} \cdots z^{A_{k}}\cdot T^A_{A_{k} \cdots A_{2} A_{1}}(u),\ee
respect the condition $z^A=0$ for $w_A\ne\za\in\zone$, determining therefore canonically a submanifold $E[\za]$ in $E$, and polynomials in (\ref{transitions}) reduce to linear functions, so we get a canonical vector bundle structure $\zt^\za:E[\za]\to M$. The corresponding module of sections over $\Ci(M)$ consists of homogeneous functions of {weight } $\za$ on $E$.

\mn Similarly, for $i,j=1,\dots,n$, $i\ne j$, denote with $C^{ij}_E$ the submanifold of $E$ defined locally by
$$C^{ij}_U=\big\{(u^s,z^A)\,|\, z^A=0\ \text{if}\ w_A^i\ne w_A^j\big\}.$$
It is easy to see that it is a vector bundle over $M_{ij}=M_i\cap M_j$. It is called the \emph{$(i,j)$-core} of $E$. The corresponding module over $\Ci(M_{ij})$ consists of homogeneous functions which are simultaneously of {weight } 1 with respect to $\n^j$ and $\n^j$.
\begin{remark}
It is well known that $E$ is isomorphic as an $n$-vector bundle with the trivial $n$-vector bundle
$${ \prod_{\za\ne 0}{}_{{}_M}E[\za]}$$ %\bigoplus_{\za\ne 0}{}_{{}_M}E[\za]$$
from Example \ref{triv}, { where $E[\za]$ are defined locally by \eqref{df:E[alpha]}}, however, in a non-canonical way (cf. \cite{Bonavolonta:2013,Gracia:2009}). This can be viewed as an analog of the { Gaw\c{e}dzki-Batchelor } theorem \cite{Batchelor:1979,Gawedzki:1977} saying that any supermanifold is (noncanonically) diffeomorphic to a superized purely even vector bundle, but the proof is much easier in this case.
\end{remark}

\section{Parity reversion in $n$-vector bundles}
Let us observe that the parity reversion functor $\zP$ has its $n$-vector bundle analog (see \cite{Voronov:2012}). Namely, for an $n$-vector bundle $(E,\n^1,\dots,\n^n)$, we apply the parity reversion functor subsequently with respect to the VB-structures $\zt_n,\zt_{n-1},\dots,\zt_1$. Consequently, we obtain a new $n$-vector bundle $(\zP E,\n^1,\cdots,\n^n)$, where the Euler vector fields $\n^i$ remain formally the same but homogeneous coordinates $z^A$ change their parity. More precisely, local affine coordinates $(z^A)$ turn into $(\az)$ with the same {weight } $w_A$ and with a change in the parity: it remains the same for $|w_A|$ even and is reversed for $|w_A|$ odd. Applying the parity reversion functor successively to all vector bundle structures, we get the \emph{total parity reversion functor} $\zP$ from the category $\vbn$ of $n$-vector superbundles into $\vbn$ defined by
\be\label{nzP}\zP=\zP^1\circ\zP^2\circ\cdots\circ\zP^n,\ee
where $\zP^i$ is the parity reversion in the vector bundle $\zt_i:E\to M_i$. This makes sense, as $\zP^i E$ remains an $n$-vector bundle with formally the same commuting Euler vector fields, since the parity of $z\pa_z$ remains even if we change the parity of $z$. Since each $\zP^i$ is an equivalence functor in the category of vector bundles, $\zP$ is an equivalence functor in the category $\vbn$ of $n$-vector bundles; the inverse functor is
$$\zP^{-1}=\zP^n\circ\zP^{n-1}\circ\cdots\circ\zP^1.$$
In the case when $E$ is an $[n]$-vector bundle, i.e., the parity of $z^A$ is $|w_A|$, the total parity reversion $\zP$ turns $E$ into a purely even $n$-vector bundle $\zP E$. This way, we get the following.
\begin{proposition} The functor (\ref{nzP}) is an equivalence functor from the category $\vbn$ of $n$-vector bundles into the same category. Reduced to the full subcategory $\vbne$ of $[n]$-vector bundles,
$\zP$ induces an equivalence of the category $\vbne$ with the category $\cln$ of purely even $n$-vector bundles.
\end{proposition}
\no With every multi-affine coordinates $(u^s,z^A)$ in $E$, we can associate in an obvious way multi-affine coordinates $(u^s,\az)$ in $\zP E$, where the {weight } of $\az$ is $w_A$ but the parity $\ze(\az)$ of $\az$ is
\be\label{parity}
\ze(\az)=\ze(z^A) + |w_A|.
\ee
Note, however, that the transition maps for the new coordinates will change.

\begin{example}\label{aba} For the total parity reversion of a double vector bundle $E$ with local homogeneous coordinates $(x^\za)$, where $\za\in {\{0, 1\}^2}$ is the {weight } of $x^\za$ and the parity of $x^\za$ is $|\za|$, we get another double vector bundle with local homogeneous coordinates $(x^\za_\zp)$. The {weight } of $x^\za_\zp$ is still $\za$,  but the parity is purely even. For instance, the transition map
$$(x')^{(1,1)}=x^{(1,1)}+x^{(1,0)}x^{(0,1)}$$
goes into
\be\label{ab}(x'_\zp)^{(1,1)}=x_\zp^{(1,1)}-x_\zp^{(1,0)}x_\zp^{(0,1)}.\ee
Indeed, to perform $\zP^2$ we write $x^{(1,0)}x^{(0,1)}$ as $-x^{(0,1)}x^{(1,0)}$, then we apply $\zP^2$ and get $-x_\zp^{(0,1)}x^{(1,0)}$, where $x_\zp^{(0,1)}$ is even, we transpose coordinates and get $-x^{(1,0)}x_\zp^{(0,1)}$, and finally we perform $\zP^1$.
\end{example}

\mn It is clear that we can do a similar total reversion for an arbitrary order of subsequent parity reversions: for $\zs$ being a permutation from the symmetric group $\Sn$, we put \be\label{nzP1}\zP^\zs=\zP^{\zs(1)}\circ\zP^{\zs(2)}\circ\cdots\circ\zP^{\zs(n)}.\ee
In other words,
$$\zP^\zs =P^{\,\zs^{-1}}\circ\zP\circ P^{\,\zs}$$
is another functor $\zP^\zs:\vbn\to\vbn$. We will show that the $n$-vector bundles $\zP E$ and $\zP^\zs E$ are generally different, but canonically isomorphic. If in Example \ref{aba} we combine (\ref{ab}) with the sign change for coordinates of {weight } $(1,1)$, i.e., with the multiplication by sign in the core bundle $C^{12}_E$, we get a map
$$\zF:\zP^1\zP^2 E\to \zP^2\zP^1 E$$
which acts as the identity on  $x^{(0,0)}_\pi$, $x^{(1,0)}_\pi$, and $x^{(0,1)}_\pi$, and maps $x^{(1,1)}_\pi$ into $-x^{(1,1)}_\pi$, which is consistent with the transition map, i.e., is an isomorphism of double vector bundles, however, different from the identity.

\mn It is easy to see (cf. \cite{Voronov:2012}) that reversing the order of two subsequent parity reversions results in an isomorphism on $n$-vector bundles
\be\label{zF}\zF^{\zs_0}_E:\zP E^{\zs_0}\to\bl\zP E\br^{\zs_0},\ee
where $\zs_0=(i,i+1)\in\Sn$ is the transposition of $i$ and $(i+1)$, given by
$$\bl\az\br_{\zs_0}\circ\zF^{\zs_0}_E=(-1)^{w_{A}(i)w_A(i+1)}\cdot\bl z^A_{\zs_0}\br_\pi\,.
$$
{ (Recall that $\pi$ relates to the total parity reversion in $n$-vector bundles, and that we follow the definitions of the symbols $\az$ and $z^A_{\zs}$ given next to \eqref{parity} and \eqref{wzs1}. In particular, $(u^s_{\zs_0}, \bl \az \br_{\zs_0})$ (respectively, $(u^s_{\zs_0}, \bl z^A_{\zs_0} \br_\pi)$) are multi-affine coordinate systems on $\bl \zP E \br^{\zs_0}$ (respectively, on $\zP E^{\zs_0}$) inherited from $(u^s, z^A)$ on the $n$-vector bundle $E$.)}
Indeed, if we consider a homogeneous monomial
\be\label{pol}W_\za(u,z)=z^{A_{1}} z^{A_{2}} \cdots z^{A_{k}}T(u),\ee
of {weight } $\za\in\zone$, then there is no difference between applying $\zP^1\circ\zP^2$ and $\zP^2\circ\zP^1$ to $W_\za$ in the case when at least one of $\za(1),\za(2)$ is 0. If, however, $\za(1)=\za(2)=1$ and, say, $w_{A_1}(1)=w_{A_2}(2)=1$, then we go with $z^{A_{2}}$ to the front to change the parity and back, but in one case we go over $z^{A_{2}}$, while over $z^{A_2}_\zp$ in the other, that produces the difference by sign. This justifies the fact that (\ref{zF}) respects the transition maps. Composing transpositions, we get a similar $\zF^\zs_E$ for all $\zs\in\Sn$:
\be\label{zF2}
\zF^{\zs}_E:\zP E^{\zs}\to\bl\zP E\br^{\zs},  \quad \zs\in\Sn.
\ee
The recipe is that changing the order of adjacent elements in the {weight } $w_A$ of $z^A_\pi$ gives a sign change any time we exchange two 1. Composing transpositions, we get the corresponding sign change for the resulted permutation, in our case $\zs^{-1}$: if we consider the subsequence of all 1 in $w_A\in\zone$, then the corresponding sign change of the coordinate $z^A_\pi$, denoted $\sgn(w_A,\zs)$, is the number of transpositions of adjacent $1$ in the process of composing $\zs$ out of adjacent transpositions transferring $w_A$ to ${w}_A^\zs$. In other words,
$$\bl\az\br_{\zs}\circ\zF^{\zs}_E=\sgn(w_A,\zs)\cdot\bl z^A_{\zs}\br_\pi,
$$
where %(cf. \cite{Heuer:2024})
\be\label{zsA}\sgn(\za,\zs)=\prod_{\substack{\za(i)=\za(j)=1\\ i<j}}\sgn\Bl\frac{\zs(j)-\zs(i)}{j-i}\Br\, =  \prod_{\substack{i<j, \\ \sigma(j)>\sigma(i)}} (-1)^{\za(i) \za(j)}
\ee
denotes the { Koszul } sign. In particular,
$$\bl\az\br_\zs\circ\zF^{\zs}_E=\sgn(\zs)\cdot\bl z^A_\zs\br_\pi$$
if $w^\zs_A=w_A$. It is well known that
\be\label{sign}\sgn(\za,\zs'\zs)=\sgn(\za^{\zs'},\zs)\cdot\sgn(\za,\zs'),\ee
where
\be \label{e:Sn-action}
\za\mapsto \za^{\zs}= \za\circ \zs = \bl\za^{\zs(1)},\dots,\za^{\zs(n)}\br
\ee
denotes  the right action of the symmetric group $\Sn$ on $\zone$.
Note that
\begin{align*}
&\bl z^A_\pi\br_{\zs'\zs}\circ P^\zs\bl\zF^{\zs'}_{E}\br\circ\zF^\zs_{E^{\zs'}}=\Bl\bl z^A_{\zs'}\br_{\zs}\Br_\pi\circ P^\zs\bl\zF^{\zs'}_{E}\br\circ\zF^\zs_{E^{\zs'}}=\sgn(w_A,\zs')\Bl\bl z^A_{\zs'}\br_\pi\Br_{\zs}\circ\zF^\zs_{E^{\zs'}}\\
&=\sgn\bl w_A^{\zs'},\zs\br\sgn\bl w_A,\zs'\br\Bl\bl z^A_{\zs'}\br_\zs\Br_\pi
=\sgn\bl w_A,\zs'\zs\br\bl z^A_{\zs'\zs}\br_\pi\,,
\end{align*}
so
\be\label{zfcomp}\zF^{\zs'\zs}_E= P^{\zs}\bl\zF^{\zs'}_{E}\br\circ\zF^{\zs}_{E^{\zs'}}.
\ee
In the short-hand notation, $\zF^{\zs'\zs}_E=\zF^{\zs'}_{E}\circ\zF^{\zs}_{E^{\zs'}}$.
On the diagram,
\be\label{ZFco}\zP(E^{\zs'\zs})=\zP\bl(E^{\zs'})^\zs\br\overset{\zF^{\zs}_{\! E^{\zs'}}}
{\xrightarrow{\hspace{1cm}}}\bl\zP E^{\zs'}\br^\zs\overset{P^{\zs}(\zF^{\zs'}_{E})}{\xrightarrow{\hspace{1.5cm}}}
\bl( \zP E)^{\zs'})^{\zs}=\bl\zP E\br^{\zs'\zs}.
\ee
Moreover, it is easy to see that the transformation $\zF^\zs_E$ commutes with the induced morphisms, or is `natural' in the categorical sense (an isomorphism of functors): if $\zf:E_1\to E_2$ is a morphism of $n$-vector bundles, then the diagram
\be\label{zs}
\xymatrix@C+45pt@R+25pt{
\zP(E^\zs_1)\ar[r]^{\zP(P^\zs\zf)}\ar[d]^{\zF^{\zs}_{E_1}} & \zP(E^\zs_2)\ar[d]^{\zF^{\zs}_{E_2}} \\
(\zP E_1)^\zs\ar[r]^{P^\zs(\zP\zf)}& (\zP E_2)^\zs}
\ee

\mn is commutative. Hence, we can view $\zF^\zs$ as an isomorphisms of functors
\be\label{zP2}\zF^\zs:\zP\circ P^\zs\to P^\zs\circ\zP,\ee
{ given by the family of $n$-vector bundle isomorphisms $\zF^\zs_E$, cf. \eqref{zF2}.}
In this way, we get the following generalization of \cite[Proposition 1]{Voronov:2012}.
\begin{proposition}\label{p1}
For every $\zs\in\Sn$ there is a canonical isomorphism of functors (\ref{zP2})
commuting with the induced morphisms, $\zF^{\zs}\circ \zP(P^{\zs}\zf)=P^\zs(\zP\zf)\circ\zF^{\zs}$.
In the short-hand notation (see Remark \ref{zs0}),
\be\label{Fzspi}\zF^{\zs}\circ \zP(\zf)=\zP(\zf)\circ\zF^{\zs}.
\ee
In multi-affine coordinates, the isomorphism $\zF^\zs_E:\zP(E^\zs)\to(\zP E)^\zs$ takes the form
$$ {u}^s_\zs\circ\zF^{\zs}_E={u}^s_\zs,\quad
\bl\az\br_\zs\circ\zF^{\zs}_E=\sgn(w_A,\zs)\cdot\bl z^A_\zs\br_\pi.
$$
\end{proposition}

\section{Iterated tangent bundles}
The canonical association of the tangent bundle with a supermanifold gives rise to the \emph{tangent functor} $\sT$ from the category of supermanifolds into the category of vector superbundles. Another important feature of the tangent bundle is that we can lift tensor fields from $M$ to $\sT M$. This works exactly like in the purely even case \cite{Grabowski:1995,Yano:1973}.

Let $Y$ be a vector field on a supermanifold $M$. The \emph{tangent lift} $\dt Y$ is a vector field on $\sT M$ uniquely determined by the identity (cf. \cite{Grabowska:2024})
\be\label{tl1} \dt Y(\zi_\zm)=\zi_{\Ll_Y\zm}\,,\ee
where $\zm$ is an arbitrary 1-form on $M$, $\zi_\zm$ is the linear function on $\sT M$ represented by $\zm$, and $\Ll_Y=\xd\circ i_Y+i_Y\circ\xd$ denotes the Lie derivative along $Y$.
In the adapted local coordinates $(x^a,\dot x^b)$ on $\sT M$, the tangent lift of a vector field $Y=\sum_af^a(x)\pa_{x^a}$ on $M$ is given by
\be\label{lift1} \dt Y=\sum_a\left(f^a(x)\pa_{x^a}+\left(\sum_b\dot x^b\frac{\pa f^a}{\pa x^b}(x)\right)\,\pa_{\,\dot x^a}\right)\,.\ee
An important fact is that the tangent lift of vector fields respects the Lie bracket \cite{Grabowski:1995}.

We will mainly use the tangent (complete) lifts $\dt\n_E$ of Euler vector fields $\n_E$ which are characterized very simply: as Euler vector fields are even and complete, the tangent lift $\dt\n_E$ is the generator of the flow $\sT\zf_{\n_E}^s$ of diffeomorphisms of $\sT E$, where $\zf^s_{\n_E}$ is the flow of the Euler vector field $\n_E$ on a vector { bundle } $E$. In the adapted coordinates for affine coordinates $(u^s,z^A)$ in $E$, the tangent lift of the Euler vector field $\n_E=z^A\pa_{z^A}$ reads
\be\label{tl}\dt\n_E=z^A\pa_{z^A}+\dot z^A\pa_{\dot z^A}.
\ee
Since $\sT\zf_{\n_E}^s$ is a VB-morphism of $\sT E\to E$, the tangent lift is a linear vector field on $\sT E\to E$, i.e., it commutes with the Euler vector field of the vector { bundle } $\sT E\to E$,
\be\label{tlc}[\dt\n_E,\n_{\sT E}]=0.\ee
It is now clear that the tangent bundle $\sT E$ of a vector bundle $\zt:E\to M$ is canonically a double vector bundle with the diagram
\be\label{dvb}
\xymatrix@C+35pt@R+15pt{
\sT E \ar[r]^{\zt_E}\ar[d]_{\sT\zt} & E\ar[d]^{\zt} \\
\sT M \ar[r]^{\zt_M} & M\,.}
\ee
The commuting Euler vector fields are
$$\n_{\sT E}=\dot x^a\pa_{\dot x^a}+\dot y^i\pa_{\dot y^i}\quad \text{and}\quad \dt\n_E= y^i\pa_{ y^i}+\dot y^i\pa_{\dot y^i}.
$$
A particular example is the iterated tangent bundle $\sT\sT M$ whose diagram reads
\be\label{dvbt}
\xymatrix@C+35pt@R+15pt{
\sT\sT M \ar[r]^{\zt_{\sT M}}\ar[d]_{\sT\zt_M} & \sT M\ar[d]^{\zt_M} \\
\sT M \ar[r]^{\zt_M} & M\,.}
\ee
The corresponding Euler vector fields we will denote $\n^1_M=\n_{\sT(\sT M)}$ and $\n^2_M=\dt\n_{\sT M}$. But this DVB, canonically associated with the manifold $M$, has an additional important property: both VB structures are canonically isomorphic.
\begin{proposition}
Let $M$ be a supermanifold. Then, there is a canonical diffeomorphism $\zk_M:\sT\sT M\to\sT\sT M$ such that $\zk_M^2=\id$ and $\zk_M$ intertwines the two vector bundle structures on $\sT\sT M$,
$$(\zk_M)_*(\n_M^1)=\n^2_{M}.$$
In adapted coordinates,
\be\label{zk}
\bl x^a,\dot x^b,\zd x^c,\zd\dot x^d\br\circ \zk_M=\bl x^a,\zd x^b,\dot x^c,\zd\dot x^d\br.
\ee
In particular, $\zk$ is the identity on the \emph{core} vector bundle $C\to M$ given by
$$C=\big\{\bl x^a,\dot x^b,\zd x^c,\zd\dot x^d\br\,|\, \zd x^b=\dot x^c=0\big\}\subset\sT\sT M.$$
\end{proposition}
\begin{proof}
The transition maps in the adapted coordinates read
$$y^a=y^a(x),\quad \dot y^a=\dot x^b\frac{\pa y^a}{\pa x^b}(x),\quad \zd y^a=\zd x^b\frac{\pa y^a}{\pa x^b}(x),
\quad \zd\dot y^a=\zd\dot x^b\frac{\pa y^a}{\pa x^b}(x)+\zd x^b\dot x^c\frac{\pa^2 y^a}{\pa x^c\pa x^b}(x),$$
and respect the local form (\ref{zk}) of $\zk$. Note that the partial derivatives $\frac{\pa^2 y^a}{\pa x^c\pa x^b}(x)$ are super-symmetric (Schwarz Theorem for supermanifolds).

\end{proof}
\no The diffeomorphism $\zk$ is called the \emph{canonical flip} of $\sT\sT M$. Fixed points of $\zk$, i.e., the submanifold $\dot x^a-\zd x^a=0$, are called \emph{holonomic vectors} and form the second tangent bundle $\sT^2M$ (second jets of curves) of $M$.  In what follows, we will use the convention in which $x^a$ is denoted with {$x^{a, (1,0)}$, $\zd x^a$ with $x^{a, (0,1)}$, and $\zd\dot x^a$ with $x^{a, (1,1)}$. In this notation, $\zk$ is the transposition $x^{a, (i,j)}\circ\zk=x^{a, (j, i)}$.
}

\iffalse
$x^a_{[0,0]}$, $\dot x^a$ is denoted with $x^a_{[1,0]}$, $\zd x^a$ with $x^a_{[0,1]}$, and $\zd\dot x^a$ with $x^a_{[1,1]}$. In this notation, $\zk$ is the transposition $x^a_{[i,j]}\circ\zk=x^a_{[j,i]}$.
\fi

\mn It is clear that, inductively, we obtain an $n$-vector bundle structure on the $n$-iterated tangent bundle $\sT^{(n)}M=\sT\sT\cdots\sT M$. The commuting Euler vector fields are
$$\n_M^k=\dt^{k-1}\n_{\sT(\sT^{(n-k)}M)},\quad k=1,\dots,n,$$
where $\dt^l$ is the iteration of $l$ subsequent tangent lifts, $\dt^0=\id$. In particular, $$\n_M^1=\n_{\sT(\sT^{(n-1)}M)}\quad\text{and}\quad \n_M^{n}=\dt^{n-1}\n_{\sT M}.
$$
Note that the subsequent flips $\zk_{\sT^{(j)}M}$, $j=0,\dots,n-2$, we can extend canonically to $\sT^{(n)}M$ by putting
$$\zk^j_M=\sT^{(n-j-2)}\bl\zk_{\sT^{(j)}M}\br.$$
They act as automorphisms of the $n$-vector bundle structure.
Moreover, composing these flips, we obtain a left action $\zs\mapsto I^\zs$ of the symmetry group $\Sn$ on the supermanifold $\sT^{(n)}M$, which permutes the vector bundle structures (the Euler vector fields)
$$I^\zs_*(\n_M^i)=\n_M^{\zs(i)}.$$
In other words, $\Sn$ acts on $\sT^{(n)}M$ by diffeomorphisms being isomorphism of $n$-vector bundle structures
\be\label{Izs}I^\zs:\sT^{(n)}M\to \bl{\sT}^{(n)}M\br^\zs
\ee
and
$$I^{\zs'\zs}=I^{\zs'}\circ I^{\zs}.$$
on the level of diffeomorphisms. On the level of morphisms of $n$-vector bundles, we could write more precisely
\be\label{Izs1}P^\zs(I^{\zs'})\circ I^{\zs}=I^{\zs'\zs},
\ee
We have the adapted homogeneous coordinates {$\bl x^{a, (\za)}\br$, } where $\za\in\zone$ is the {weight } of $x^{a, (\za)}$.
Since any subsequent flip is the transposition of elements of $\za$, for any $\zs\in\Sn$ we have (cf. (\ref{wzs1}))
\be\label{zs3}x^{a, (\za)} \circ I^\zs=x^{a, (\za^\zs)},
\ee
Of course, for the transposition $\zs_0=(i,j)$, the diffeomorphism $I^{\zs_0}$ acts as the identity on the $(i,j)$-core $C^{ij}_{\sT^{(n)}M}$. In this way, we get a canonical model of a \emph{symmetric $n$-vector bundle} \cite{Bruce:2016a}.

\begin{definition}
A \emph{symmetric $n$-vector bundle} is an $n$-vector bundle $\bl E,\n^1,\dots,\n^n\br$ equipped with a left action $I$ of the symmetric group $\Sn$ by diffeomorphisms $I^\zs$ of $E$, $\zs\in\Sn$, permuting accordingly the Euler vector fields, $I^\zs_*(\n^i)=\n^{\zs(i)}$, and such that for any transposition $\zs=(i,j)$ the diffeomorphism $I^\zs$ acts as the identity on the $(i,j)$-core $C^{ij}_E$.

\mn A \emph{morphism} in the category $\svbn$ of symmetric $n$-vector bundles is a morphism of $n$-vector bundles intertwining the action of the symmetry group $\Sn$. For  $[n]$-vector bundles, the corresponding subcategory is denoted by $\svbne$.
\end{definition}
%begin{MR}
\begin{lemma}\label{l:adapted_coord}
Any symmetric { $n$-vector } bundle  $E = (E; \n^1, \ldots, \n^n)$ admits a system of graded coordinates $(z_\alpha^A)$ such that
for any $\sigma\in\Sn$ we have
%\footnote{$\alpha.\sigma =(\alpha_{\sigma(1)}, \ldots, \alpha_{\sigma(n)})$, $(\alpha, \sigma)\mapsto \alpha.\sigma$ is a right action of $\Sn$ on $\{0, 1\}^n$.}
\begin{equation}\label{e:coord_sym}
    z^A_\alpha \circ I^{\zs} = z^A_{\za^\zs},
\end{equation}
 where $\za^\zs$ is as in \eqref{e:Sn-action}, the index $A$ ranges over a set $\Omega$ given as a disjoint union $\Omega = \bigsqcup_{j=0}^n \Omega_j$, and $\alpha \in \{0,1\}^n$ encodes the weight of the coordinate $z_\alpha^A$. The total weight $|\alpha|$ of $z_\alpha^A$ is equal to $j$ whenever $A \in \Omega_j$.
It follows that the number of coordinates $z_\alpha^A$ of weight $\alpha$ is equal to $\sharp \Omega_j$, where $j=|\alpha|$, and this number is the same for all $\alpha$ with the same total weight.
\end{lemma}

\begin{proof}
We will refer to a coordinate system $(z^A)$  on a symmetric $n$-vector bundle satisfying \eqref{e:coord_sym} as a \emph{nice} coordinate system.

It is instructive to first consider the case $n = 2$. Assume that the symmetric structure on a double vector bundle $(E; \n^1, \n^2)$ (in the category of supermanifolds) is given by an involution $I$, and let $C$ denote the core of $E$.
Choose graded coordinates $(x^a; y^i, Y^j; z^\mu)$ of weights $00$, $01$, $10$, and $11$, respectively, and such that $Y^j = I^*(y^j)$, so that $I^*(Y^j) = y^j$, since $I$ is an involution. By assumption, $I|_C = \id_C$, hence
\[
I^*(z^\mu) = z^\mu + y^i Y^j I^\mu_{ji}(x)
\]
for some functions $I^\mu_{ji}$ defined locally on a supermanifold $M$, the total base of $E$. From $I^*(I^*(z^\mu)) = z^\mu$, we get
\[
I^\mu_{ji} + (-1)^{\tilde{i}\tilde{j}} I^\mu_{ij} = 0,
\]
i.e., the functions $\left(I^\mu_{ij}\right)$ are (super) skew-symmetric in the lower indices.
We aim to show that every symmetric double vector bundle admits a nice coordinate system, meaning in our case that all the functions $I^\mu_{ij}$ vanish. We seek functions $Q^\mu_{ji}(x)$ such that
\[
(x^a; y^i, Y^j; \und{z}^\mu = z^\mu + y^i Y^j Q^\mu_{ji}(x))
\]
yields a nice system. We have
$$
I^*(\und{z}^\mu) = z^\mu + y^i Y^j \left(I^\mu_{ji}(x) + (-1)^{\tilde{i}\tilde{j}} Q^\mu_{ij}(x)\right),
$$
so $I^*(\und{z}^\mu) = \und{z}^\mu$  is satisfied when we choose
\[
Q^\mu_{ij}(x) := \frac{1}{2} I^\mu_{ij}(x).
\]
Note that $\und{z}^\mu_{ij} = \frac12 \bl z^\mu_{ij} + I^*(z^\mu_{ij}) \br$.

Now consider the general case $n \geq 2$.
% over on open subset $U \subset M$
 Let $(y^A)$ be any system of $\{0, 1\}^n$-graded coordinates on $E$ defined over an open sub-supermanifold $U \subset M$, where $M$ is the total base  of $E$. Without loss of generality, we may assume that $M=U$.

Fix $1\leq m\leq n$. Consider weights $\alpha\in\{0, 1\}^n$ such that $|\alpha| = m$, and let $\zb = (\underbrace{1, \ldots, 1}_{m}, 0, \ldots, 0)$ denote one such weight.
 Let $\Omega_m$ be the set of indices $A$ of the coordinates $(y^A)$ of weight $\zb$.
Denote
$$
\Stab(\alpha) = \{\sigma\in \Sn : \alpha^\sigma = \alpha\},
$$
the stabilizer of a weight $\alpha$.  Note that $\Stab(\alpha)$ is a subgroup of $\Sn$ isomorphic with $\Sgroup_m \times \Sgroup_{n-m}$, so $ \# \Stab(\alpha) =  m!(n-m)!$.
We correct the coordinates $(y^A)$ of weight $\zb$ by setting
\be \label{df:zA0}
    z^A_{\zb} := \frac{1}{\sharp  \Stab(\zb)} \sum_{\sigma\in \Stab(\zb)} y^A \circ I^\sigma, \quad  A\in \Omega_m.
\ee
Then   for any $\sigma\in \Stab(\zb)$  and $A\in \Omega_m$, we have $z^A_{\zb} \circ I^\sigma = z^A_{\zb}$ as $\Stab(\zb)$ is a group.

  Recall that there is a non-canonical isomorphism of $n$-vector bundles, $E \simeq \prod_{M} E[\alpha]$, the fiber product over $M$ of the vector bundles $E[\alpha]$, indexed by $\alpha \in \{0, 1\}^n\setminus 0^n$, where $E[\alpha] \subset E$ are the vector bundles defined in \eqref{df:E[alpha]}.  For each $\sigma\in \Stab(\alpha)$ we have
\begin{equation}\label{e:gVeps}
 I^\sigma|_{E[\alpha]} = \id_{E[\alpha]}.
\end{equation}
 Indeed, any $\sigma\in \Stab(\alpha) \simeq \mathbb{S}_m\times \mathbb{S}_{n-m}$ is a product of transpositions from $\Stab(\alpha)$, and for each such transposition $\sigma =(i, j)$, the equality \eqref{e:gVeps} holds due to assumption that $I^\sigma$  acts as the identity on the core bundle of $(E; \n^i, \n^j)$.

 Therefore, from  \eqref{e:gVeps},
 $$
    z^A_{\zb}|_{E[\zb]} = y^A|_{E[\zb]},
 $$
for any $A\in \Omega_m$ and $z^A_{\zb}$ is homogenous of weight $\zb$. Since $(y^A)_{A\in \Omega_m}$ is a frame of local coordinates of weight $\zb$, the same is true for the family $(z^A_{\zb})_{A\in \Omega_m}$.

Now take any $\alpha$ with $|\alpha| = m$ and consider the set $G_\alpha:= \{\sigma\in \Sn : \zb^\sigma=\alpha\}$. It coincides with $\Stab(\zb)\cdot \sigma$ where $\sigma$ is any element from the non-empty set $G_\alpha$. Define
\begin{equation} \label{d:z^a_eps}
    z_\alpha^A = \frac{1}{m! (n-m)!} \sum_{\sigma\in G_\alpha} z_{\zb}^A \circ I^\sigma,
\end{equation}
where $A\in \Omega_m$, and $z_{\zb}^A$ are defined by \eqref{df:zA0}. Note that $z_{\zb}^A \circ I^\sigma$ are homogenous of weight $\alpha = \beta^\sigma$, hence $z_\alpha^A$ is the same. Besides, for any $\tau\in \Sn$ we have $z_{\alpha}^A \circ I^\tau = z_{\alpha.\tau}^A$. Indeed,
$$
    z_\alpha^A \circ I^\tau = \frac{1}{m! (n-m)!} \sum_{\sigma\in G_\alpha} \left( z_{\zb}^A \circ I^{\sigma\tau} \right) = \frac{1}{m! (n-m)!} \sum_{\sigma\in G_{\alpha.\tau}} z_{\zb}^A \circ I^\sigma = z^A_{\alpha^\tau},
$$
as $\sigma\in G_{\alpha}$ if and only if $\sigma^\tau\in G_{\alpha^\tau}$.
Hence, $(z^A_{\alpha})_{A\in \Omega_m}$ is a  frame of local coordinates of weight $\alpha$, since $I^\sigma: E[\alpha] \to E[\zb]$ is a vector bundle isomorphism, where $\sigma$ is such that $\zb^\sigma = \alpha$.

 Thus, the union of the functions $\left(z^A_\alpha\right)$ defined in \eqref{d:z^a_eps}, where $\alpha$ varies in $\{0, 1\}^n\setminus 0^n$, and $A\in \Omega_{|\alpha|}$, together with base coordinates $\bl z^A_{0^n} = y^A\br$, where $\w(y^A) = 0^n$, constitutes a coordinate system for the symmetric vector bundle $E$ satisfying the equation \eqref{e:coord_sym}.
\end{proof}

\no We can view $I^\zs$ as an isomorphism $I^\zs:E\to E^\zs$ of $n$-vector bundles with
$$z^A_\zs\circ I^\zs=z^A.$$
It is also natural to write $I^{\zs'}$ also for the isomorphism
$$I^{\zs'}=P^\zs(I^{\zs'}):E^\zs\to \bl E^{\zs'}\br^\zs=E^{\zs'\zs},$$
since it is represented by the same diffeomorphism of $E$. Obviously, we have the following.
\begin{proposition} The $n$-vector bundle $E=\sT^{(n)}M$ associated with a manifold $M$ is a canonical example of a symmetric $n$-vector bundle.
\end{proposition}

\section{Parity reversion for symmetric $n$-vector bundles}
Let us now apply the total parity reversion functor $\zP$ to a symmetric $n$-vector bundle $\bl E,\n^1,\dots,\n^n, I\br$. Since in this case, for any $\zs\in\Sn$ we have an isomorphism $I^\zs:E\to\es$, thus
$$\zP I^\zs:\zP E\to \zP\es,$$
composing $\zP I^\zs$ with the isomorphisms of $n$-vector bundles (cf. (\ref{zP2}) %\commentMRR{warto dodać jeszcze odnośnik do definicji $\zF^\zs_E$})
$\zF^\zs_E:{\zP E^\zs}\to\bl\zP E\br^\zs$, we get an isomorphism of $n$-vector bundles
\be\label{A}J^\zs_E=\zF^\zs_{E}\circ\zP I^\zs:\zP E\to\bl\zP E\br^\zs.\ee
Note, however, that for a transposition $\zs_0=(i,j)$ the isomorphisms $I^{\zs_0}$, thus $\zP I^{\zs_0}$, acts as the identity on the $(i,j)$-core $C^{ij}_E$. { Due to Lemma~\ref{l:adapted_coord}, there exists homogenous coordinates $(u^s, z^A)$ on $E$ such that }
$$z^A_\zp\circ { \zP} I^\zs=z^A_\zp\quad \text{if}\quad \zs(w_A)=w_A.$$
As $\zF^{\zs_0}_E$ is the multiplication by $-1$ on the core bundle $C^{ij}_E$, the isomorphism $J^{\zs_0}_E$ acts as the multiplication by $-1$ on $C^{ij}_{\zP E}$.
Hence, like for $\zF^\zs_E$, we have
\be\label{Aij} {u}^s_\pi\circ J^{\zs}_E={u}^s_\pi,\quad
\bl\az\br_\zs\circ J^{\zs}_E=\sgn(w_A,\zs)\cdot z^A_\pi\,;
\ee
in particular,
$$\bl\az\br_\zs\circ J^{\zs}_E=\sgn(\zs)\cdot z^A_\pi$$
if $\zs(w_A)=w_A$.
From (\ref{zs}), we get the commutative diagram
$$
\xymatrix@C+35pt@R+25pt{
\zP(E^\zs)\ar[r]^{\zP P^\zs(I^{\zs'})}\ar[d]^{\zF^{\zs}_{E}} & \zP(E^{\zs'\zs})\ar[d]^{\zF^{\zs}_{E^{\zs'}}} \\
(\zP E)^\zs\ar[r]^{P^\zs(\zP I^{\zs'})}& (\zP E^{\zs'})^{\zs}\,,}
$$
and further (cf. (\ref{zfcomp}))
\begin{align*}& J^{\zs'\zs}_E=\zF^{\zs'\zs}_E\circ \zP I^{\zs'\zs}=P^{\zs}\bl\zF^{\zs'}_{E}\br\circ\zF^{\zs}_{E^{\zs'}}\circ \zP P^\zs(I^{\zs'})\circ\zP I^\zs\\
&\ \ =P^{\zs}\bl\zF^{\zs'}_{E}\br\circ  P^\zs(\zP I^{\zs'})\circ\zF^{\zs}_{E^{\zs'}}\circ\zP I^\zs
=P^{\zs}\bl J^{\zs'}_E\br\circ J^\zs_{E^{\zs'}}\,;
\end{align*}
in the short-hand notation,
\be\label{J} J^{\zs'\zs}_E=J^{\zs'}_E\circ J^\zs_{E^{\zs'}}.
\ee
Viewing $J^\zs_E$ as just a diffeomorphisms $J^\zs$ of $\zP E$, we can write simply $J^{\zs'\zs}=J^{\zs'}\circ J^\zs$, i.e., we deal with a left $\Sn$-action on $\zP E$.
If $\zf:E_1\to E_2$ is a morphism of symmetric $n$-vector bundles, then $\zf\circ I^\zs_1=I^\zs_2\circ\zf$, and from (\ref{zs}) we get
\be\label{zs2}
\xymatrix@C+45pt@R+25pt{
\zP E_1\ar[r]^{\zP I^\zs_1}\ar[dr]^{J^\zs_{E_1}}\ar@/^3pc/[rrr]^{\zP\zf}&\zP(E^\zs_1)\ar[r]^{\zP(P^\zs\zf)}\ar[d]^{\zF^{\zs}_{E_1}} & \zP(E^\zs_2)\ar[d]_{\zF^{\zs}_{E_2}} &\zP E_2\ar[l]_{\zP I^\zs_2}\ar[dl]_{J^\zs_{E_2}}\\
& (\zP E_1)^\zs\ar[r]^{P^\zs(\zP\zf)}& (\zP E_2)^\zs\,,}
\ee

\mn so on the level of diffeomorphisms
$$\zP\zf\circ J^\zs_{E_1}=J^\zs_{E_2}\circ\zP\zf.$$
This object looks like a symmetric $n$-vector bundle, but for this $\Sn$-action any transposition $\zs=(i,j)$ acts on the $(i,j)$-core $C^{ij}_{\zP E}$ as the multiplication by $-1$. In this way, the concept of a \emph{skew-symmetric $n$-vector bundle} is born (cf. \cite{Bruce:2016a}).
\begin{definition}
A \emph{skew-symmetric $n$-vector bundle} is an $n$-vector bundle $\bl E,\n^1,\dots,\n^n\br$ equipped with a left action $J$ of the symmetry group $\Sn$ by diffeomorphisms $J^\zs$ of $E$, $\zs\in\Sn$, permuting accordingly the Euler vector fields, $J^\zs_*(\n^i)=\n^{\zs(i)}$, and such that for any transposition $\zs=(i,j)$ the diffeomorphism { $J^\zs$} acts as the multiplication by -1 on the $(i,j)$-core vector bundle $C^{ij}_E$.

\no A \emph{morphism} in the category $\ssvbn$ of symmetric $n$-vector bundles is a morphism of $n$-vector bundles intertwining the action of the symmetry group $\Sn$.
\end{definition}
\no We know already that $\zP E$ is purely even for any $[n]$-vector bundle $E$. The corresponding subcategory in $\ssvbn$ we will denote $\ssln$. Its objects are purely even skew-symmetric $n$-vector bundles.

\mn{
From \eqref{Aij} and \eqref{J} it follows that if $\bl E,\n^1,\dots,\n^n, I\br$ is a symmetric $n$-vector bundle, then $\bl \zP E,\zP\n^1,\dots,\zP \n^n, \zF\circ \zP I\br$ is a skew-symmetric $n$-vector bundle, and vice versa. (A vector field $X$ on a vector bundle $E \to M$ of weight~$0$ has its superization $\zP X$ on $\zP E \to M$—this rule can be iteratively applied to the total parity reversion functor and weight vector fields.) Therefore, a skew-symmetric $n$-vector bundle also admits homogeneous coordinates $(u^s, z^A)$ such that $J^\zs$ permutes the homogeneous coordinates, and if a coordinate $z^A$ is of weight $w_A$, then the coordinate $z^A \circ J^{\zs}$ is of weight $\zs(w_A)=\bl w_A^{\zs(1)},\dots,w_A^{\zs(n)}\br$, and such that
\[
z^A \circ J^\zs = \sgn(\zs)\cdot z^A \quad \text{if} \quad \zs(w_A)=w_A.
\]
}
What immediately follows from our considerations is the following.
\begin{theorem}\label{main}
There exists a canonical equivalence functor $\Xi$ between the category $\svbn$ of symmetric $n$-vector bundles and the category $\ssvbn$ of skew-symmetric $n$-vector bundles, given explicitly by the total parity reversion,
$$\Xi(E,I)=(\zP E,J),\quad \Xi(\zf)=\zP\zf,$$
where $J^\zs=\zF^\zs\circ\zP I^\zs$ satisfies \eqref{Aij} and \eqref{J}. This equivalence restricts to an equivalence functor on the level of the category $\svbne$ of symmetric $[n]$-vector bundles on one hand, and the category $\ssln$ of skew-symmetric purely even $n$-vector bundles on the other.
\end{theorem}

\section{Polarizations and desuperization of $\bN$-manifolds}

It is easy to see that any $n$-vector bundle $\bl E,\n^1,\dots,\n^n\br$ is canonically also an $\N$-weighted bundle with respect to the weight vector field $\n=\n^1+\cdots+\n^n$. A \emph{polarization} of an $\N$-weighted bundle of degree $n$ is, roughly speaking, a canonical description of an inverse process. This will lead to an equivalence of categories $\nbn$ and $\svbn$ and, consequently, $\nbbn$ and $\svbne$. Composed with the equivalence functor $\Xi$ between symmetric  and skew-symmetric $n$-vector bundles, this also gives the equivalence of categories $\nbn$ and $\ssvbn$, which we call the \emph{reverse polarisation}.

\mn One particular case is of special interest, namely the case of $\bN$-manifolds and $[n]$-vector bundles. Here, the skew-symmetric $n$-vector bundles associated with $\bN$-manifolds are purely even, i.e., they are objects in the traditional purely even differential geometry.
Let $(E,\n)$ be an $\N$-weighted bundle of degree $n$, with local homogeneous coordinates $(u^s,z^A)$ and the weight vector field
$$\n=\sum_Aw_Az^A\pa_{z^A},$$
where $w_A$ are positive integers. We have then the adapted coordinates $(u^s_\za,z^A_\zb)$, $\za,\zb\in\zone$, in the symmetric $n$-vector bundle $\sT^{(n)}E$ with the commuting Euler vector fields
$$\n_E^k=\dt^{k-1}\n_{\sT(\sT^{(n-k)}E)},\quad k=1,\dots,n,$$
as well as the higher tangent lift $\dt^n\n$. All these vector fields are even, complete, and pairwise commuting. It is easy to see that
$$\dt^n\n=\sum_{A,\za}w_Az^A_\za\pa_{z^A_\za}$$
and
$$\n_E^k=\sum_{\za(n-k+1)=1}\Bl\sum_s u^s_\za\pa_{u^s_\za}+\sum_A z^A_\za\pa_{z^A_\za}\Br,$$
so
$$\hat\n=\sum_{k=1}^n\n_E^k=\sum_{\za}|\za|\Bl\sum_s u^s_\za\pa_{u^s_\za}+\sum_A z^A_\za\pa_{z^A_\za}\Br.$$
Now, consider a submanifold $E^{(n)}$ in $\sT^{(n)}E$ given by $\dt^n\n=\hat\n$, i.e.,
\be\label{cond}u^s_\za=0\quad\text{if}\quad |\za|>0\quad\text{and}\quad z^A_\za=0\quad\text{if}\quad |\za|\ne w_A.
\ee
The definition is correct, since the transition maps respect these conditions.
Moreover, the Euler vector fields $\n^k_E$ are tangent to $E^{(n)}$, and conditions (\ref{cond}) are invariant with respect their permutations, so the $\Sn$-action on $\sT^{(n)}E$ reduces to $E^{(n)}$ and turns $E^{(n)}$, with the restricted Euler vector fields $\n_E^k$, into a symmetric $n$-vector bundle. This is exactly the construction introduced in \cite{Bruce:2016a}, applied without any changes to the case of supermanifolds.
\begin{definition}
The symmetric $n$-vector bundle $\bl E^{(n)},\n^1,\dots,\n^n\br$ we call the \emph{polarization} of the $\N$-weighted bundle $E$ of degree $n$.
\end{definition}

\no Since $\sT$ is a functor in the category of supermanifolds, $\sT^{(n)}$ is also such a functor and it is easy to see that in the case when $\zf:E_1\to E_2$ is a morphism of $\N$-weighted bundles of degree $n$, the map $\sT^{(n)}\zf$ restricts to a morphism $\zf^{(n)}:E_1^{(n)}\to E_2^{(n)}$ between symmetric $n$-vector bundles.

\mn Conversely, if $\bl F,\n^1,\dots,\n^n\br$ is a symmetric $n$-vector bundle with an $\Sn$-action permuting the Euler vector fields, then the {\emph{the diagonalization of $F$ } (see \cite{Bruce:2016a})  is defined as the { submanifold } $F^{\Sn}$ %$E$
of fixed points. That is,
$F^{\Sn}$ locally defined in  $\{0, 1\}^n$-graded coordinates $(z^A)$ on $F$ by equations $z^A \circ I^\sigma = z^A$ for all $\sigma\in \Sn$. In  a nice coordinate system $\bl z^A_{\alpha}\br$ on $F$ as described in  Lemma~\ref{l:adapted_coord}, we have
$$F^{\Sn} = \{\bl z^A_{\alpha} \br:   z^A_\alpha = z^A_\beta \text{ whenever } |\alpha| = |\beta| \}.$$
 It is easy to see that $\bl E^{(n)}\br^{\Sn}$, the diagonalization of $E^{(n)}$, is canonically isomorphic with $E$ as $\N$-weighted bundles. The embedding,
 denoted by $\diag: E \to E^{(n)} \subset \sT^{(n)} E$, giving rise to this isomorphism is given by
$$
    z^{A, (\za)} \circ \diag = |\za| \cdot z^A,
$$
where $\w(z^A) = |\za|$, and $\bl  z^{A, (\za)} \br_{\w(z^A) = |\za|}$ is the coordinate system on $E^{(n)}\subset \sT^{(n)} E$. It is clear that the definition of $\diag$ does not depend on the choice of coordinate system $(z^A)$.
}
This way, we get the following.
\begin{theorem}\label{m1}
The associations $\cP^{(n)}E=E^{(n)}$ and $\cP^{(n)}\zf=\zf^{(n)}$ define a functor $\cP^{(n)}$ from the category $\nbn$ of $\N$-weighted bundles of degree $n$ to the category $\svbn$ of symmetric $n$-vector bundles.

This functor, called the \emph{polarization functor}, is an equivalence of categories. The appropriate restriction of this functor yields an equivalence functor from the category of $\nbbn$ of $\bN$-manifolds of degree $n$ into the category $\svbne$ of symmetric  $[n]$-vector  bundles.
\end{theorem}
\no If we confront the above result with Theorem \ref{main}, we immediately obtain an equivalence functor $\fL=\Xi\circ\cP$ { from } the category $\nbn$ of $\N$-weighted bundles of degree $n$ into the category $\ssvbn$ of skew-symmetric $n$-vector bundles. If our $\N$-weighted bundle is a $\bN$-manifold, then $E^{(n)}$ is a symmetric $[n]$-vector bundle, and $\Xi(E^{(n)})$ is purely even $n$-vector bundle. Of course, $\fL=\zP\circ \cP^{(n)}$, being the composition of two equivalence functors, is an equivalence functor from the category $\nbn$ of $\N$-weighted bundles into the category $\ssvbn$ of skew-symmetric $n$-vector bundles. We call $\fL$ the \emph{reverse polarization functor}. This way, we get the following.
\begin{theorem}\label{m2}
The functor $\fL=\Xi\circ\cP^{(n)}$ of the reverse polarization is a canonical equivalence functor from the category $\nbn$ of $\N$-weighted bundles of degree $n$ into the category $\ssvbn$ of skew-symmetric $n$-vector bundles. Reduced to the category $\nbbn$ of $\bN$-manifolds of degree $n$, it gives an equivalence functor from $\nbbn$ into the category  $\ssln$ of purely even skew-symmetric $n$-vector bundles, called the \emph{desuperization functor}.
\end{theorem}
\no Note that the last part of the above theorem recovers the main result by Heuer and Jotz \cite{Heuer:2024}.
\begin{remark} Let $E$ be an $\N$-weighted bundle as above. Then, one easily verifies that
$$
\diag(E) = \sT^n E \cap E^{(n)},
$$
where $\sT^n E \subset \sT^{(n)} E$ denotes the $n$-th order tangent  bundle of the  supermanifold $E$. In other words, the diagonalization of  $E^{(n)}$, being the image of the map $\diag$, is the submanifold of holonomic elements in $E^{(n)}$.  For a geometric perspective on this construction -- where the functor $\sT^n$  is defined via superized versions of curves and their jets using the functor of points approach -- see \cite{Bruce:2014}.
\end{remark}

\small{\vskip.7cm}

\noindent Katarzyna GRABOWSKA\\
Faculty of Physics, University of Warsaw\\
ul. Pasteura 5, 02-093 Warszawa, Poland
\\Email: konieczn@fuw.edu.pl\\
https://orcid.org/0000-0003-2805-1849\\

\noindent Janusz GRABOWSKI\\ Institute of
Mathematics,  Polish Academy of Sciences\\ ul. \'Sniadeckich 8, 00-656 Warszawa, Poland
\\Email: jagrab@impan.pl\\
https://orcid.org/0000-0001-8715-2370\\

\noindent Miko\l aj ROTKIEWICZ\\ Institute of
Mathematics,  University of Warsaw\\ ul. Banacha 2, 02-097 Warszawa, Poland
\\Email: mrotkiew@mimuw.edu.pl\\
https://orcid.org/0000-0002-9510-0463

\end{document}